\documentclass[reqno,12pt,twoside,english]{article}

\usepackage{ amsmath, amsfonts, amssymb, amsthm, amscd,color}
\usepackage[T1]{fontenc}
\usepackage[english]{babel}

\newtheorem{theorem}{Theorem}

\newtheorem{proposition}{Proposition}

\newtheorem{definition}{Definition}

\begin{document}
\author{
Yuchao Dong\\{\small Universit\'e d'Angers,}
	\\{\small D\'epartement de Math\'ematiques,}
	\\{\small 2, Bd Lavoisier,}\\
	{\small 49045 Angers Cedex 01, France,}\\ 
\and
Lioudmila Vostrikova\\
	{\small  Universit\'e d'Angers,}\\
	{\small  D\'epartement de Math\'ematiques,}\\
	{\small  2, Bd Lavoisier,}\\
        {\small 49045 Angers Cedex 01, France,}\\
}
\title{
Utility maximization for L\'evy switching models} 
\maketitle
\vspace{-1cm}
\begin{abstract} 
This article is devoted to the maximisation of HARA utilities of L\'evy switching process
on finite time interval via dual method. We give the description of all $f$-divergence minimal martingale measures in initially enlarged filtration, the expression of their Radon-Nikodym densities involving Hellinger and Kulback-Leibler processes, the expressions of the optimal strategies in progressively enlarged filtration for the maximisation of HARA utilities  as well as the values of the corresponding maximal expected utilities. The example of Brownian switching model is presented to give the financial interpretation of the results. 
\end{abstract}

\section{Introduction}
\par  Regime switching processes are the processes whose parameters depend on some Markov process, taking a finite number of values. Those processes have been applied successfully  since more than 20 years to model the behaviour of economic time series and financial time series. Hamilton and al.  in \cite{Ham} considered regime switching autoregressive conditional heteroskedastic models and proposed a way for their calibration. Cai in \cite{Cai}
considered regime switching ARCH models. So and al. in \cite{SLL} considered the so-called Markov switching volatility models to capture the changing of the behaviour of the volatility due to the economic factors.
\par L\'evy switching processes  was widely used for pricing due to their flexibility to reflect the real data and their comodity  as mathematical tool. For two state variance gamma model, Komkov and Madan in \cite{KM} have developed a method based on fast Fourier Transform to price the vanila options.
Later, using the implementation of numerical methods of resolving PDE equations, Chordakis \cite{Cho} has priced the exotic contracts, like barrier, Bermuda and American options.
Then, Elliott and Osakwa in  \cite{EO} used multi-state pure jump processes to price vanila derivatives. Later, Elliott and al. \cite{ESC} considered Heston Markov switching model to price variance swap and volatility swap with probabilistic and PDE approaches. Hainaut \cite{Hai} gave an overview of the tools related to L\'evy  switching models, like moments and the conditions under which such processes are martingales, the measure transform preserving L\'evy  switching structure, and also the questions of option pricing and calibration.
\par The question of the optimal hedging in discrete time methodology that minimizes the expected value of  given penalty function of the hedging error with general regime-switching framework was studied by So and al. \cite{SLL}. Francois and al. \cite{FGG} considered the question of the portfolio optimisation in affine models with Markov switching. Escobar and al. \cite{ENZ} derive the optimal investment strategies for the expected utility maximization of terminal wealth.
\par The utility maximization problem of the L\'evy switching models can be solved by the so-called  dual method. The dual method for utility maximization has been deeply studied by Goll and Ruschendorf in \cite{GR} in general semi-martingale setting. This method permits to find a semi-explicit  NA conditions as well as a semi-explicit expression of the optimal (asymptotically optimal) strategies for utility maximization. This method was applied for pricing and hedging of exponential L\'{e}vy models in Miyahara \cite{M}, Fujiwara and Myahara \cite{FM}, Chouli and al. \cite{CS}, \cite{CSL}, Essche and Schweizer\cite{ES}, Hubalek and Sgarra \cite{HS}, Jeanblanc and al.\cite{JKM}, Cawston and Vostrikova \cite{CV1}, \cite{CV2}. This method was also efficient to determine the indifference prices in Ellanskaya and Vostrikova \cite{EV}, Vostrikova \cite{V}. Then, in Cawston and Vostrikova \cite{CV3} the  problem of utility maximization was solved for change-point exponential L\'evy  models. The L\'evy switching models are natural generalisations of change-point exponential L\'evy  models, permitting multiple changes of the behaviour of L\'evy  processes on the observed interval of time. 
\par  This article is devoted to the utility maximisation problem  for L\'evy switching
models. More precisely, we consider $N$ independent L\'evy processes $X^{(j)}=(X_t^{(j)})_{0\leq t\leq T}, j=1,2,...,N,$ on the interval $[0,T]$, with values in $\mathbb R^d$ and starting from $0$, which are defined on a probability space $(\Omega_1,\mathcal G,P_1)$ with the natural filtration $\mathbb G=(\mathcal G_t)_{0\leq t\leq T}$ satisfying  usual conditions. Each L\'evy process $X^{(j)}$ has the   characteristic triplet denoted by $(b^{(j)},c^{(j)},\nu^{(j)})$ where $b^{(j)}$ is the drift parameter, $c^{(j)}$ is the quadratic variation of its continuous martingale part and $\nu^{(j)}$ is L\'evy measure  which verifies the usual condition
\[\int_{\mathbb R^d}(\|x\|^2\land1)\nu^{(j)}(dx)<\infty\]
where $\|\cdot\|$ is euclidean norm in $\mathbb{R}^d$.
 \par As known, the law of $X^{(j)}$ is entirely characterized by the characteristic function $\varphi_t^{(j)}$ of $X_t^{(j)}$ with $t>0$, which is, for $\lambda \in \mathbb R^d$, given by
\[\varphi_t^{(j)}(\lambda)=\exp\{t\psi^{(j)}(\lambda)\}\]
where $\psi^{(j)}$  is  the characteristic exponent of $X^{(j)}$. We recall that $\psi^{(j)}$ is  determined by the L\'evy-Khinchin formula, namely
\[\psi^{(j)}(\lambda)=i\langle\lambda,b^{(j)}\rangle-\frac{1}{2}\langle\lambda,c^{(j)}\lambda\rangle+\int_{\mathbb R^d}\left (e^{i\langle\lambda,x\rangle}-1-i 
\langle\lambda,x\rangle I_{\{\|x\|\le 1\}}\right )\nu^{(j)}(dx)\]
where $I(\cdot)$ denotes the indicator function.
\par Let also $\alpha=(\alpha_t)_{{0\leq t\leq T}}$ be a homogeneous Markov process with the values in the set $\{1,2,3,...,N\}$, given on the probability space $(\Omega_2,\mathcal H,P_2)$ with the natural filtration $\mathbb H=(\mathcal H_t)_{t \ge 0}$ such that $\mathcal H_0=\{\emptyset,\Omega\}$, which is independent from the L\'{e}vy processes $X^{(1)},X^{(2)},...,X^{(N)}$. 
\par Then on the product space   $(\Omega,\mathcal F,\mathbb P)=(\Omega_1\times\Omega_2,\mathcal G \times \mathcal H,P_1\times P_2)$
we can define two filtrations.  One is the progressively enlarged filtration $\mathbb F=(\mathcal F_t)_{0\leq t\leq T}$ where  for $t<T$
\[\mathcal F_t=\bigcap_{s>t}\mathcal G_s \otimes \mathcal H_s \,\,\,\mbox{and}\,\,\,\mathcal{F}_T=\mathcal G_T \otimes \mathcal H_T\]
and the second one is the initially enlarged filtration $\hat {\mathbb F}=(\hat {\mathcal F}_t)_{0\leq t\leq T}$ defined  for $t<T$ as
$$\hat {\mathcal F}_t=\bigcap_{s>t}\mathcal G_s \otimes \mathcal H_T \,\,\,\mbox{and}\,\,\,\hat{\mathcal{F}}_T=\mathcal G_T \otimes \mathcal H_T.$$
For technical reasons we will work  in the initially enlarged filtration, and surprisingly at the first glance, we will obtain the results in progressively enlarged filtration for minimal martingale measures and the optimal strategies for utility maximisation. Such phenomenon can be explained by the fact that firstly both filtrations coincide at the time $T$, and secondly, by  the fact that
 the L\'evy processes $X^{(j)}, j=1,\cdots N,$ remain independent from the Markov process $\alpha$ under any  minimal martingale measure.
\par On the  defined above product space, we  now introduce a L\'evy switching process $X$ such that the increments of this process coincide with the increments of $X^{(j)}$ when the process $\alpha$ states in the state $j$:  
\begin{equation}\label{definition switching}
dX_t=\sum_{j=1}^NdX_t^{(j)}I_{\{\alpha_{t-}=j\}}
\end{equation}
or
\[X_t=\sum_{j=1}^N\int_0^tI_{\{\alpha_{s-}=j\}}dX_s^{(j)}.\]
More explicitly, if, for example, at $t=0$, $\alpha_0=i_0$ and $\tau_1$ is the first time of change from the state $i_0$ to another state $i_1$
and $\tau_k=\inf\{t>\tau_{k-1}|\alpha_t \neq i_{k-1}\}$ for $k \ge 1$, then
\[X_t=\left \{{
\begin{split}
&X_t^{(i_0)},\text{ for } t\le \tau_1,\\
&X_{\tau_1}^{(i_0)}+X_t^{(i_1)}-X_{\tau_1}^{(i_1)},\text{ for } \tau_1 <t\le \tau_2,\\
&...\\
&X_{\tau_n}^{(i_{n-1})}+X_t^{(i_n)}-X_{\tau_{n}}^{(i_n)},\text{ for } \tau_n <t\le \tau_{n+1},\\
&....
\end{split}
}\right .\]
The characteristic function of $X$ can be find easily since $X$ is a process with, conditionally to $\alpha$, independent increment. Due to the mutual independence of L\'evy processes and $\alpha$, we have
\begin{equation*}
\begin{split}
E_{\mathbb P}\left[\exp(i\langle\lambda,X_t\rangle)\right]&=E_{\mathbb P}\left[E_{\mathbb P}\left[\exp(i\langle\lambda,X_t\rangle)|\alpha\right]\right]\\
&=E_{\mathbb P}[E_{\mathbb P}[\exp(i\langle\lambda,\sum_{j=1}^N\int_0^tI_{\{\alpha_{s-}=j\}}dX_s^{(j)}\rangle)|\alpha]]\\
&=E_{\mathbb P}[\prod_{j=1}^NE_{\mathbb P}[\exp(i
\langle\lambda,\int_0^tI_{\{\alpha_{s-}=j\}}dX_s^{(j)}\rangle)|\alpha]].\\
\end{split}
\end{equation*}
We remark that for any real-valued deterministic function $q=(q_s)_{s\ge0}$ such that $\int_0^tq_sdX_s^{(j)}$ exists,  the characteristic function verifies
\begin{equation*}
\begin{split}
&E_{\mathbb P}[\exp(i\langle\lambda,\int_0^tq_sdX_s^{(j)}\rangle)]
=\exp(i\langle\lambda,b^{(j)}\rangle\int_0^tq_sds-\frac{1}{2}\langle \lambda, c^{(j)}\,\lambda \rangle\int_0^tq_s^2ds+\\&\hspace{4cm}\int_0^t\int_{\mathbb R^d}\left (e^{i\langle\lambda q_s,x\rangle}-1-i\langle\lambda q_s,x\rangle I_{\{\|x\|\le 1\}}\right )\nu^{(j)}(ds,dx).
\end{split}
\end{equation*}
If $q_s=I_{\{\alpha_{s-}=j\}}$ for $s \ge 0$, then $q_s^2=q_s$ and if $I_{\{\alpha_{s-}=j\}}=1$ then 
\[\int_{\mathbb R^d}\left (e^{i\langle\lambda q_s,x\rangle}-1-i\langle\lambda q_s,x\rangle I_{\{\|x\|\le 1\}}\right )\nu^{(j)}(dx)=\int_{\mathbb R^d}\left (e^{i\langle\lambda ,x\rangle}-1-i\langle\lambda,x\rangle I_{\{\|x\|\le 1\}}\right )\nu^{(j)}(dx).\]
In addition, if $I_{\{\alpha_{s-}=j\}}=0$ then this integral is equal to $0$. Finally,
\[E_{\mathbb P}[\exp(i\langle\lambda,X_t\rangle)]=E_{\mathbb P}[\exp(\sum_{j=1}^N \,\psi^{(j)}(\lambda)\int_0^tI_{\{\alpha_{s-}=j\}}ds)].\]
\par We will consider the problem of the utility maximization for the process $S=(S_t)_{t\ge 0}$ such that the components of $S$ denoted by $S^{(k)},1\le k \le d$, are Dol\'eans-Dade exponentials of the corresponding components of $X$, denoted by $\bar{X}^{(k)}$, i.e. for all $t\geq 0$
\[S^{(k)}_t=S^{(k)}_0\exp\{\bar{X}_t^{(k)}-\frac{1}{2}\langle \bar{X}^{(k),c}\rangle_t\}\prod_{0<s\le t}e^{-\Delta \bar{X}_s^{(k)}}(1+\Delta \bar{X}^{(k)}_s)\]
where $\bar{X}^{(k),c}$ is continuous martingale part of $\bar{X}^{(k)}$ and $\langle \bar{X}^{(k),c}\rangle$ is its predictable quadratic variation.
\par As utility functions, we consider HARA utilities, which are logarithmic, power and exponential utilities, defined as
\begin{equation*}
\begin{split}
&u(x)=\ln(x) \text{ with } x >0,\\
&u(x)=\frac{x^p}{p}\text{ with } x>0 \text{ and } p \in (-\infty,0)\cup(0,1),\\
&u(x)=1-\exp(-x) \text{ with } x\in \mathbb R.\\
\end{split}
\end{equation*}

Let us denote by $\mathcal A$ a set of self-financing admissible strategies. We recall that an admissible strategy is a predictable process $\Phi=(\eta,\phi)$ taking values in $\mathbb R^{d+1}$ where $\eta$ represents the quantity invested in the non-risky asset $B$ and $\phi=(\phi^{(1)},\phi^{(2)},...,\phi^{(d)})$ are the quantities invested in the risky assets $S^{(1)},S^{(2)},...,S^{(d)}$ respectively such that $\eta$ is $B$-integrable and there exists  $ a \in \mathbb R^+$ such that for $t\in[0,T]$
\[\sum_{k=1}^d\int_0^t\phi_s^{(k)}dS_s^{(k)}\ge -a.\]
We say that a strategy $\hat \phi \in \mathcal A$ is $u$-optimal on $[0,T]$ if 
$$E[u(x_0+\sum_{k=1}^d\int_0^T\hat{\phi}_s^{(k)}dS_s^{(k)})]=\sup_{\phi \in \mathcal A}E[u(x_0+ \sum_{k=1}^d\int_0^T\phi_s^{(k)}dS_s^{(k)})]$$
where $x>0$ is initial capital.
A sequence of admissible strategies $(\hat \phi^n)_{n\ge1}$ is said to be asymptotically $u$-optimal on $[0,T]$ if 
\[\lim_{n \rightarrow \infty}E[u(x_0+\sum_{k=1}^d\int_0^T\hat{\phi}_s^{(k),n}dS_s^{(k)})]=\sup_{\phi \in \mathcal A}E[u(x_0+\sum_{k=1}^d\int_0^T\phi_s^{(k)}dS_s^{(k)})].\]
\par The article is organized in the following way. In Section 2 we give a short description of the known results on dual method for the utility maximisation of L\'evy processes (see Propositions 2 and 3). In Section 3 we summarize the useful information on Hellinger and Kulback-Leibler processes (see Proposition 4). In  Section 4 we give a description of all $f$-divergence minimal martingale measures for HARA utilities in progressively enlarged filtration (see Propositions 5,6 and Theorem 1). Then, in Propositions 7, 8, 9 of Section 5 we give the expressions for the optimal strategies in progressively enlarged filtration and we calculate the value of the maximal expected utility. In section 6 we  apply our results to Brownian switching model  and we give the financial interpretation of our results.
\section{Dual approach for utility maximisation of exponential L\'evy models}

\par The idea of this method is to replace the problem of utility maximization by the problem of minimization of the corresponding $f$-divergence over the set of all, equivalent to the law  $\mathbb{P}_T$ of the process $X$,  martingale measures $\mathbb{Q}_T$  where $T$ is time horizon.
The function $f$ is nothing else as the dual function of $u$, which can be obtained
by Fenchel-Legendre transform :
\[f(y)=\sup_{x\in\mathbb R}(u(x)-xy)\]
Simple calculations show that
\begin{equation*}
\begin{split}
&f(x)=-\ln(x)-1,x>0\text{ if $u$ is logarithmic},\\
&f(x)=-\frac{p-1}{p}x^{\frac{p}{p-1}},x>0\text{ if $u$ is power},\\
&f(x)=1-x+x\ln(x),x>0\text{ if $u$ is exponential}.\\ 
\end{split}
\end{equation*} 
We recall that for two equivalent measures $\mathbb{Q}_T$ and $\mathbb{P}_T$ the $f$-divergence of $\mathbb{Q}_T$ w.r.t. $\mathbb{P}_T$ is defined as
$$f(\mathbb{Q}_T|\mathbb{P}_T)= E_{\mathbb{P}}\left[f\left(\frac{d\mathbb{Q}_T}{d\mathbb{P}_T}\right)\right].$$
\par We mention now some useful definitions  and results on the $f$-divergence problem and optimal investment problem in exponential L\'evy models.
For that, we take a L\'evy process $L=(L_t)_{0\leq t\leq T}$   with the values in $\mathbb{R}^d$,  given on probability space $(\Omega,\mathcal{F},P)$ with the filtration $\mathbb{F}$ satisfying usual properties. We denote the characteristic triplet of the process $L$ by $(b_L,c_L,\nu_L)$
where $b_L$ is drift parameter, $c_L$ is the predictable variation of the continuous martingale part of $L$, and $\nu _L$ is L\'evy measure. We denote by $\mathbb{P}$ the law of $L$ and we suppose that $L$ is integrable, i.e. $E_{\mathbb{P}}(|L_t|)<\infty$ for 
$t\in[0,T]$.

\begin{definition}
We say that $\mathbb{Q}_T^{*}$ is an $f$-divergence minimal equivalent martingale measure for L\'evy process $L=(L_t)_{0\leq t\leq T}$  if
\begin{enumerate}
\item the measure $\mathbb{Q}_T^{*}$ is equivalent to $\mathbb{P}_T$, i.e. $\mathbb{Q}_T^{*}\sim \mathbb{P}_T$,
\item the process $L$ is a martingale w.r.t. $(\mathbb{F},\mathbb{Q}_T^{*})$,
\item $f(\mathbb{Q}_T^{*}|\mathbb{P}_T) < +\infty $ and
	$$f(\mathbb{Q}^{*}_T|\mathbb{P}_T)=\inf_{\mathbb{Q}_T\in \mathcal {M} }f(\mathbb{Q}_T|\mathbb{P}_T)$$
	where  $\mathcal{M}$ is the set of all equivalent martingale measures $\mathbb{Q}_T$.
	\end{enumerate}
\end{definition}
\begin{definition}
		We say that an $f$-divergence minimal martingale measure $\mathbb{Q}_T^*$ is invariant under scaling if for all $x \in \mathbb R^+$,
		\[f(x\mathbb{Q}_T^*|\mathbb{P}_T)=\inf_{\mathbb{Q}_T\in \mathcal{M}}f(x\mathbb{Q}_T|\mathbb{P}_T)\]
\end{definition}

\begin{definition}
		We say that an $f$-divergence minimal martingale measure $\mathbb{Q}_T^*$ is time horizon invariant if for all $t \in]0,T]$
		\[f(\mathbb{Q}_t^*|\mathbb{P}_t)=\inf_{\mathbb{Q}_t\in \mathcal{M}}f(\mathbb{Q}_t|\mathbb{P}_t)\]
\end{definition}

\begin{definition}
 We say that an $f$-divergence minimal martingale measure $\mathbb{Q}_T^*$ preserves the L\'evy property if $L$ remains a L\'evy process under $\mathbb{Q}_T^*$.
\end{definition}

\begin{proposition}(cf. \cite{CV1},\cite{CV2}) Let $L$ be an integrable L\'evy process, $f$ be one of the three dual functions above and  $\mathbb{Q}^{*}_T$  the corresponding $f$-divergence minimal martingale measure. Then this measure is  time horizon invariant, it is invariant under scaling and it preserves L\'evy property of the process~$L$. 
\end{proposition}
From Girsanov theorem (see \cite{JSh}, Ch. 3, Theorem 3.24, p.159) we can also easily get the following. 
\begin{proposition}The Radon-Nikodym process  $\mathbb{Z^*}=(\mathbb{Z}^*_t)_{0\leq t\leq T}$ of the measure $\mathbb{Q}^{*}_T$ w.r.t. $\mathbb{P}_T$ verify:
$$\mathbb{Z}_t^*=\frac{d\mathbb{Q}^{*}_t}{d\mathbb{P}_t}= \mathbb{Z}_0^*\,\mathcal{E}(m)_t$$
where $\mathcal{E}(\cdot)$ is Dol\'eans-Dade exponential and
$m=(m_t)_{0\leq t\leq T}$ is a martingale with
$$m_t= \int_0^t\,^{\top}\!\beta ^* dY^c_s+ \int_0^t\int_{\mathbb R^d}(Y^{*}(x)-1)(\mu_{L}-\nu_{L})(ds,dx)$$
where $\mu_{L}$ and $\nu_{L}$ are jump measure of the process $L$ and its compensator under $(\mathbb{P},\mathbb{F})$.
The so-called Girsanov parameters $(\beta^*, Y^*)$ for the change of the measure $\mathbb{P}_T$ into $\mathbb{Q}_T^*$  are independent of $(\omega, t)$.
Moreover,  the drift $b^{\mathbb{Q}}_L$ of the process $L$ under $\mathbb{Q}^{*}_T$ is:
$$b^Q_L=b_L+c_L\beta^{*}+\int_{\mathbb R^d}x\cdot (Y^{*}(x)-1)\nu_L(dx)$$
\end{proposition}

\begin{proposition}\label{GR}(cf.\cite{GR},\cite{CV1}) Let $\mathbb{Z}^*_T$ be Radon-Nikodym derivative of $f$-divergence minimal martingale measure $\mathbb{Q}^*_T$ with respect to $\mathbb{P_T}$. Let $x_0>0$ be the initial capital. We suppose that 
for  $\lambda _0>0$ 
such that
\begin{equation}\label{initial}
-E_{\mathbb{P}}(\mathbb{Z}^*_T\,f'(\lambda_0 \mathbb{Z}^*_T))=x_0
\end{equation}
we have
$$E_{\mathbb{P}}(|f(\lambda_0 \mathbb{Z}^*_T)|)<\infty, \,\,\,E_{\mathbb{P}}(\mathbb{Z}^*_T\,|f'(\lambda_0 \mathbb{Z}^*_T)|)<\infty . $$ 
Then  there exists an optimal (respectively asymptotically optimal) strategy $\hat{\phi}$ such that
$$-f'(\lambda _0\mathbb{Z}^*_T)=x_0+\sum_{k=1}^d\int_0^T  \hat{\phi}^{(k)}_s dS^{(k)}_s$$
where $(\int_0^{\cdot}  \hat{\phi}^{(k)}_s dS^{(k)}_s)$  are $\mathbb{Q}^*_T$-martingales, $1\leq k\leq d$.
If we choose $\hat{\phi}^0_t$ such that
$$B_t\,\hat{\phi}^0_t= x_0+  \sum_{k=1}^d \int_0^t  \hat{\phi}^{(k)}_s dS^{(k)}_s-\sum_{k=1}^d \hat{\phi}^{(k)}_t S^{(k)}_t$$
then the strategy $\Phi=(\hat{\phi}^0,\hat{\phi}^{(1)},\cdots,\hat{\phi}^{(d)})$ is self-financing and it is optimal for logarithmic and power utility and asymptotically optimal for exponential utility.
Moreover, the maximal expected utility  $$U_T(x_0)=E_{\mathbb{P}}[u(-f'(\lambda _0\mathbb{Z}^*_T))].$$
\end{proposition}
\section{Hellinger integrals, Kulback-Leibler informations and the corresponding processes}
We recall here some useful results on Hellinger integrals and Kulback-Leibler informations and corresponding processes. As known, these notions was introduced in semi-martingale setting and one can
find the details about these processes in \cite{JSh}, \cite{Ko}.  For convenience of the readers we will restrict ourself to the case of the processes with independent increments.
\par  Let $V$ be the process with independent increments which is a semi-martingale without predictable jumps, observed on the interval $[0,T]$.  We denote the triplet of the semi-martingale characteristics by $(B, C, \nu )$ where $B$ stands for the drift part, $C$ is the predictable variation of the continuous martingale part, and $\nu$ is the compensator of the jumps of the process $V$.
We denote $P_T$ the law of $V$ which is entirely defined by the triplet $(B, C, \nu )$ (see for the details \cite{JSh}, Ch.2, p.114).
\par Let $Q_T\ll P_T$ and  $Z_T=\frac{dQ_T}{dP_T}$ with the Girsanov parameters $(\beta,Y)$. 
We recall that the Hellinger integral of order $\gamma\in ]0,1[$ of the measure $Q_T$ w.r.t. the measure $P_T$
$${\bf H}^{(\gamma)}(Q_T\,|\,P_T)= E_{\mathbb P}\left [Z_T^{\gamma}\right ]$$
and this definition can be extended for $\gamma <0$ when the integral exists.  We define the Hellinger process $H{(\gamma)}=(H_t
{(\gamma )})_{0\leq t\leq T}$ of the order $\gamma$ via the expression
$$H_t(\gamma)= \frac{1}{2}\gamma (1-\gamma)\int_0^t\, ^{\top}\beta_s\, dC_s\,\beta_s-\int_0^t\int_{\mathbb{R}^d}\left [Y^{\gamma}_s(x)-\gamma Y_s(x)-1+\gamma\right ]\nu(ds,dx)$$
 The Kulback-Leibler information 
of the measure $Q_T$ w.r.t. the measure $P_T$ is defined as 
$${\bf K}(Q_T\,|\,P_T)=E_{\mathbb P}\left [Z_T\,\ln(Z_T)\right ]$$
\par
The Kulback-Leibler process $K=(K_t)_{0\leq t\leq T}$ is defined as
$$K_t= \frac{1}{2}\int_0^t\, ^{\top}\beta_s\,dC_s\,\beta+\int_0^t\int_{\mathbb{R}^d}[Y_s(x)\ln(Y_s(x) )-Y_s(x)+1]\nu(ds,dx)$$
\begin{proposition}\label{HK} (see \cite{EV}) For the process with independent increments $V$ and Girsanov parameters $(\beta, Y)$ of the change of the measure $P$ into $Q$, which   depend only on (t,x), we have:
$${\bf H}(Q_T\,|\,P_T)=\exp(-H_T(\gamma )) ,$$
$${\bf K}(Q_T\,|\,P_T)=K_T.$$
\end{proposition}

\section{F-divergence minimal martingale measures}
First of all we describe the set of all equivalent martingale measures living on the space of c\`adl\`ag functions $(D^2([0,T]),\mathcal{D}^2([0,T]))$ of the trajectories of the couple of the processes $(X,\alpha)$. Let $\mathbb{P}_T$ denote the law of $(X,\alpha)$ and let $P_T^X$ and $P_T^{\alpha}$ are the laws of $X$ and $\alpha$ respectively, on $[0,T]$. We denote also by $P_T^X(\cdot\,|\, \alpha)$ the regular conditional law of $X$ given $\alpha$.
\begin{proposition}The law $\mathbb{Q}_T\ll \mathbb{P}_T$ if and only if there exists a regular conditional law of  $X$, denoted by $Q^X_T(\cdot\,|\,\alpha)$ such that $Q^X_T(\cdot\,|\,\alpha)\ll
P^X_T(\cdot\,|\,\alpha)$ ($P^{\alpha}_T$-a.s.) and the law  of $\alpha$, denoted $Q_T^{\alpha}$, such that $Q^{\alpha}_T\ll P^{\alpha}_T$. Moreover,
$\mathbb{P}_T$-a.s.
$$\frac{d\mathbb{Q}_T}{d\mathbb{P}_T}= \frac{dQ^X_T(\cdot\,|\,\alpha)}{dP^X_T(\cdot\,|\,\alpha)}
\,\frac{dQ^{\alpha}_T}{dP^{\alpha}_T}$$
\end{proposition}
\begin{proof}
 For all $A,B\in \mathcal{D}([0,T])$
$$\mathbb{P}_T(A\times B) = \int_B P^X_T(A\,|\,\alpha=y)\,dP_T^{\alpha}(y) = 
\int_{A\times B} dP^X_T(x\,|\,\alpha=y)\,dP_T^{\alpha}(y)$$
and
$$\mathbb{Q}_T(A\times B) = \int_B Q^X_T(A\,|\,\alpha=y)\,dQ_T^{\alpha}(y) =
\int_{A\times B} dQ^X_T(x\,|\,\alpha=y)\,dQ_T^{\alpha}(y)$$
At the same time, the condition $\mathbb{Q}_T\ll \mathbb{P}_T$ is equivalent to the existence of a density $f(x,y)$ such that
$$\mathbb{Q}_T(A\times B) = \int_{A\times B}f(x,y)\,d\mathbb{P}_T(x,y) = 
\int_{A\times B}f(x,y)\, dP^X_T(x\,|\,\alpha=y)\,dP_T^{\alpha}(y)$$
In addition, obviously $\mathbb{Q}_T\ll \mathbb{P}_T$ implies the condition $Q^{\alpha}_T\ll P^{\alpha}_T$ and 
$$\mathbb{Q}_T(A\times B) = \int_{A\times B} dQ^X_T(x\,|\,\alpha=y)\,\xi_T^{(\alpha)}(y)\,dP_T^{\alpha}(y)$$
where $\xi_T^{(\alpha)}=\frac{dQ_T^{\alpha}}{dP_T^{\alpha}}$.
Since these relations hold for all $A,B\in \mathcal{D}([0,T])$ we get that $\mathbb{P}_T$-a.s.
$$\frac{d\mathbb{Q}_T}{d\mathbb{P}_T}(x,y)=f(x,y) = \frac{dQ^X_T(x\,|\,\alpha=y)}{dP^X_T(x\,|\,\alpha=y)}\,\xi_T^{(\alpha)}(y)$$
Conversely, if $Q^X_T(\cdot\,|\,\alpha)\ll
P^X_T(\cdot\,|\,\alpha)$ ($P^{\alpha}_T$-a.s.) and $Q^{\alpha}_T\ll P^{\alpha}_T$, then evidently $\mathbb{P}(A\times B)=0$ implies $\mathbb{Q}(A\times B)=0$. Since the $\sigma$-algebra $\mathcal{D}^2([0,T])$ can be generated by countable set of the events of the type $A\times B$,
it gives the claim.
\end{proof}
\par Let us suppose that the L\'evy processes $X^{(j)}, j=1,\cdots , N$, are integrable, i.e. for $t\in[0,T]$,
$E_{\mathbb{P}}\left[|X^{(j)}_t|\right]< \infty$. We denote by $\mathcal M^{(j)}$  the set of equivalent martingale measures for $X^{(j)}$. Let $Q^{(j)} \in \mathcal M^{(j)}$ and let $(\beta^{(j)},Y^{(j)})$ be the Girsanov parameters for the change of the measure $P^{(j)}$ into $Q^{(j)}$.
\par Let $\mathcal M$ denote the set of equivalent martingale measures for $X$ such that the Girsanov parameters for the change of the measure $\mathbb{P}$ into $\mathbb{Q}$ verify the conditions: 
\[\int_0^t|Y_s(x)-1|\nu_X(dx)ds<\infty\text{ and }\int_0^t\|\beta_s\|^2ds<\infty,\]
where $\nu_X$ is the compensator of the measure of jumps of the process $X$.
 We put then for all $t \ge 0$ and $x \in \mathbb R^d$
\begin{equation}
\label{form4beta}
\beta_t=\sum_{j=1}^N\beta_t^{(j)}I_{\{\alpha_{t-}=j\},}
\end{equation} 
\begin{equation}
\label{form4y}
Y_t(x)=\sum_{j=1}^N Y_t^{(j)}(x)I_{\{\alpha_{t-}=j\}.}
\end{equation}
We introduce also a process $m=(m_t)_{0\leq t\leq T}$ such that
\begin{equation}
\label{PRP_m}
m_t=\int_0^t\,^{\top}\!\beta_sdX_s^c+\int_0^t\int_{\mathbb R^d}(Y_s(x)-1)(\mu_X-\nu_X)(ds,dx)
\end{equation}
where $\mu_X$ and $\nu_X$ are jump measure of $X$ and its compensator respectively.\\
\begin{proposition}
The set $\mathcal M \neq \emptyset $ if and only if for all $j=1,2,...,N$, $\mathcal M^{(j)} \neq \emptyset$. Moreover, if $\mathbb{Q}_T\in\mathcal{M}$, the corresponding Girsanov parameters satisfy (\ref{form4beta}) and (\ref{form4y}) and
$$\frac{d\mathbb{Q}_T}{d\mathbb{P}_T}= \xi^{(\alpha)}_T\,\,\mathcal{E}(m)_T\, $$
with $m$  defined by \eqref{PRP_m} and $\xi^{(\alpha)}_T=\frac{dQ_T^{\alpha}}{dP_T^{\alpha}}$.
\end{proposition}
\begin{proof}Since the L\'evy processes $X^{(1)},X^{(2)},...,X^{(N)}$ are integrable, for $ j=1,2,...,N$, we have
\[X_t^{(j)}=a^{(j)}t+M_t^{(j)},\]
where $a^{(j)}=b^{(j)}+\int_{\mathbb R^d}xI_{\{\|x\|>1\}}\nu^{(j)}(dx)$ and 
\[M_t^{(j)}=\sqrt{c^{(j)}}W_t^{(j)} + \int_0^t\int_{\mathbb R^d}x\,(\mu^{(j)}-\nu^{(j)})(ds,dx)\]
with the symetric matrix $\sqrt{c^{(j)}}$ such that $(\sqrt{c^{(j)}})^2=c^{(j)}$, and $\mu^{(j)}$  the measure of jumps of $X^{(j)}$. We recall that $M^{(j)}$ in this case is a martingale with respect to the filtration $\mathbb G$. But, since L\'evy processes and the Markov process $\alpha$ are independent, it is also a martingale with respect to the filtration  $\hat{\mathbb F}$ (cf. \cite{CJZ}) since the immersion property holds. Moreover, from \eqref{definition switching} we get
\begin{equation*}
\begin{split}
X_t=&\sum_{j=1}^N a^{(j)}\int_0^tI_{\{\alpha_{s-}=j\}}ds+\sum_{j=1}^N\int_0^tI_{\{\alpha_{s-}=j\}}dM_s^{(j)}\\
=&A_t+M_t
\end{split}
\end{equation*}
It is clear that $M=(M_t)_{0\leq t\leq T}$ is a local martingale as a sum of integrals of bounded functions w.r.t. martingales. Since $S_t^{(k)}=\mathcal E(\bar{X}^{(k)})_t,k=1,2,...,d$, the martingale measures for $S^{(k)}$ and $\bar{X}^{(k)}$ coincide, and they are the measures which remove the drift $A=(A_t)_{0\leq t\leq T}$ of the process $X$. The canonical decomposition of $X$ in the initially enlarged filtration coincide with the canonical decomposition of $X$ conditionally to $\alpha$, as the maps on $\Omega_1\times \Omega_2$. Then the drift of $X$ under $\mathbb{Q}$ is equal as a map to the drift of $X$ under $Q^X_T(\cdot|\alpha)$.\\
Let  $Z=(Z_t)_{0\leq t\leq T}$ be the Radon-Nikodym process of an equivalent conditional measure $Q^X(\cdot|\alpha)$ w.r.t. $P^X(\cdot|\alpha)$. It can be always written as
 $$Z_t=\mathcal E(m)_t,$$ where $m=(m_t)_{t\ge 0}$  is a local martingale of the form \eqref{PRP_m}.
According to Girsanov theorem, the drift $A^{Q}$ of the process $X$ under $Q^X(\cdot|\alpha)$ is equal to:
\[A_t^Q=A_t+\int_0^tc_s\beta_sds+\int_0^t\int_{\mathbb R^d}x\cdot (Y_s(x)-1)\nu_X(dx, ds)\]
where $c=( c_s)_{s\ge0}$ is the density w.r.t Lebesgue measure of the quadratic variation of the continuous martingale part of $X$  and $\nu_X$ is the compensator of the jump measure of $X$ w.r.t. $(\mathbb{P},\hat{\mathbb F})$.\\
Let us calculate $ c_t$ and $\nu_X$. Since the continuous martingale  part $X^c$ of $X$ is equal in law to $(\sum_{j=1}^N\int_0^tI_{\{\alpha_{s-}=j\}}\sqrt{c^{(j)}}dW_s^{j})_{0\leq t\leq T}$ with independent standard Brownian motions $W^{(j)}$, the quadratic variation of $X^c$  at $t$ is equal to
\[ C_t=\sum_{j=1}^Nc^{(j)}\int_0^tI_{\{\alpha_{s-}=j\}}ds\]
and its density w.r.t. Lebesgue measure is
\begin{equation}\label{form4c}
 c_t=\sum_{j=1}^Nc^{(j)}I_{\{\alpha_{t-}=j\}}
\end{equation}
At the same time, 
\[\Delta X_t=\Delta \left(\sum_{j=1}^N\int_0^tI_{\{\alpha_{s-}=j\}}dX_s^{(j)}\right)=\sum_{j=1}^NI_{\{\alpha_{t-}=j\}}\Delta X_t^{(j)}\]
and since $I_{\{\alpha_{t-}=i\}} \cap I_{\{\alpha_{t-}=j\}}=0$ for $i \neq j$,
\begin{equation}\label{jump}\
\mu_X(dt,dx)=\sum_{j=1}^NI_{\{\alpha_{t-}=j\}}\mu^{(j)}(dt,dx)
\end{equation}
and 
\begin{equation}\label{comp}
\nu_X(dt,dx)=\sum_{j=1}^NI_{\{\alpha_{t-}=j\}}\nu^{(j)}(dt,dx)=\sum_{j=1}^NI_{\{\alpha_{t-}=j\}}\nu^{(j)}(dx)dt.
\end{equation}
Finally, 
$$A_t^{Q}=\sum_{j=1}^N \int_0^t \left[ a^{(j)}+c^{(j)}\beta_s^{(j)}+\right.
\left. \int_{\mathbb R^d}x\cdot (Y_s^{(j)}(x)-1)\nu^{(j)}(dx)\right]I_{\{\alpha_{s-}=j\}}ds.$$
\par  Thus, $X=(X_t)_{0\leq t\leq T}$ is a martingale w.r.t. $Q^X(\cdot|\alpha)$   if and only if $A_t^{Q}=0$ for all $t \ge 0$. We  take the derivative in the previous expression to obtain that
on the set $\{\alpha_{s-}=j\}$ we have  that
\[a^{(j)}+c^{(j)}\beta_t^{(j)}+\int_{\mathbb{R}^d}x\cdot(Y_t^{(j)}(x)-1)\nu^{(j)}(dx)=0.\]
As a conclusion, on the set $\{\alpha_{t-}=j\}$ the Girsanov parameters for the change of the measure $P^X(\cdot|\alpha)$  into $Q^X(\cdot|\alpha)$ , denoted $(\beta,Y)$, are equal to the Girsanov parameters $(\beta^{(j)},Y^{(j)})$ of an equivalent martingale measure for $X^{(j)}$. This proves \eqref{form4beta} and \eqref{form4y}, and we conclude using the Proposition 5. 
\end{proof}
\par 
For all $t\in[0,T]$ due to the relations \eqref{jump} and \eqref{comp}
\begin{equation}\label{mm}
m_t=\sum_{j=1}^Nm_t^{(j)}
\end{equation}
with 
\begin{equation}\label{m}
\hspace{-4cm} m_t^{(j)}=\int_0^t I_{\{\alpha_{s-}=j\}}\,^{\top}\!\beta^{(j)}\sqrt{(c^{(j)})}dW_s^{(j)}
\end{equation}
$$\hspace{4cm}+\int_0^t\int_{\mathbb R^d}I_{\{\alpha_{s-}=j\}}\,(Y^{(j)}(x)-1)(\mu^{(j)}-\nu^{(j)})(ds,dx).$$
Let us denote $Z^{(j)}_t=\mathcal E(m^{(j)}_t)$.
Now, since the processes $X^{(j)}$, $j=1,2,\cdots, N,$ are independent and cannot jump at the same time with probability 1, we conclude that
\begin{equation}\label{Z}
\begin{split}
Z_t=&\mathcal E(m)_t=\exp\{m_t-\frac{1}{2}<m^c>_t\prod_{0<s\le t}e^{-\Delta m_s}(1+\Delta m_s)\}\\
&\hspace{2cm}=\prod_{j=1}^N\mathcal E(m^{(j)}_t)=\prod_{j=1}^NZ_t^{(j)}.
\end{split}
\end{equation}
\begin{theorem}
	Suppose $\mathcal M^{(j)} \neq \emptyset$ for all $j=1,2,...,N$. Then for HARA utilities, the minimal martingale measure $\mathbb{Q}^*_T$ exists and its Radon-Nikodym derivative w.r.t. $\mathbb{P}$ is given by
	\[\mathbb{Z}_T^*=\xi_T^{(\alpha),*}\,\prod_{j=1}^NZ_T^{(j),*}\]
	where $Z_T^{(j),*}=\mathcal E(m^{(j),*})_t$ with $m^{(j),*}$ defined by  \eqref{m} with replacement of $(\beta^{(j)},Y^{(j)})$ by Girsanov parameters $(\beta^{(j),*},Y^{(j),*})$ of the minimal martingale measure and $\xi_T^{(\alpha),*}$ is the Radon-Nikodym density of an equivalent to $P^{\alpha}$ measure which minimise $f$-divergence. For logarithmic utility
$\xi_T^{(\alpha),*}=1$, for power utility it is given by the formula \eqref{power}, and for exponential utility it is defined via \eqref{exp}.
\end{theorem} 
\begin{proof} 
\par 1) First of all we prove that for  any  density $\xi^{(\alpha)}$
$$E_{\mathbb{P}}\left [f(\mathbb{Z}_T)\right ]=E_{\mathbb{P}}\left [f(\xi^{(\alpha)}_T\prod_{j=1}^NZ_t^{(j)})\right ]\geq E_{\mathbb{P}}\left [f(\xi^{(\alpha)}_T\prod_{j=1}^NZ_t^{(j),*})\right ].$$
Since the Radon-Nikodym  density process $Z=(Z_t)_{0\leq t\leq T}$ of the conditional martingale measure $Q_T^X(\cdot|\alpha)$ w.r.t. $P_T^X(\cdot|\alpha )$ is
$Z_t= \mathcal{E}(m)_t$
where $m$ is defined by \eqref{m}, and $\sum_{k=0}^{\infty}I_{\{\tau_k\leq t<\tau_{k+1}\}}=1$
for each $t>0$ and $\tau_0=0$, we get
$$ m_T= \sum_{k=0}^{\infty}\int_{\tau_k\wedge T}^{\tau_{k+1}\wedge T}\left[^{\top}\!\beta_t^{(j_k)}dX_t^{(j_k),c}+\int_{\mathbb{R}^d}(Y_t^{(j_k)}(x) -1)(\mu^{(j_k)}-\nu ^{(j_k)})(dt,dx)\right]$$
where $(\beta^{(j_k)},Y^{(j_k)})$ are the Girsanov parameter of the mentioned change of the measure for the L\'{e}vy process $X^{(j_k)}$.  The last formula means that on the set 	$\{\tau_k\leq T<\tau_{k+1}\}$ with fixed $k$
$$ m_T= \sum_{m=0}^{k-1}\int_{\tau_m}^{\tau_{m+1}}[^{\top}\!\beta_t^{(j_m)}dX_t^{(j_m),c}+\int_{\mathbb{R}^d}(Y_t^{(j_m)}(x) -1)(\mu^{(j_m)}-\nu ^{(j_m)})(dt,dx)$$
$$+\int_{\tau_k}^T[^{\top}\!\beta_t^{(j_k)}dX_t^{(j_k),c}+\int_{\mathbb{R}^d}(Y_t^{(j_k)}(x) -1)(\mu^{(j_k)}-\nu ^{(j_k)})(dt,dx)$$
and  the Radon-Nikodym density $\mathbb{Z}_T$ of the measure $\mathbb{Q}_T$ w.r.t. the measure $\mathbb{P}_T$
verifies
$$\mathbb{Z}_T=\xi^{(\alpha)}_T\,\left(\prod_{m=1}^{k-1}\frac{\xi^{(j_m)}_{\tau_m}}{\xi^{(j_m)}_{\tau_{m-1}}}\right)\,\frac{\xi^{(j_k)}_T}{\xi^{(j_k)}_{\tau_k}}$$
where $(\xi^{(j_m)})_{0\leq m\leq k},$ are Radon-Nikodym densities of the mentioned change of the measure for $(X^{(j_m)})_{0\leq m\leq k}$.
We denote  the value of $\mathbb{Z}_T$ on the set 	$\{\tau_k\leq T<\tau_{k+1}\}$  by $Z_T(k)$.
Then for any $f$-divergence corresponding to HARA utility
$$E_{\mathbb{P}}\left [f(\mathbb{Z}_T)\right ]=\sum_{k=0}^{\infty}E_{\mathbb{P}}\left [f(\mathbb{Z}_T)\,I_{\{\tau_k\leq T<\tau_{k+1}\}}\right ]=\sum_{k=0}^{\infty}E_{\mathbb{P}}\left [I_{\{\tau_k\leq T<\tau_{k+1}\}}\,E_{\mathbb{P}}\left [f(Z_T(k))\,|\alpha\right ]\right ]$$
When we take the  conditional expectation w.r.t. $\alpha$, the random variables $\xi^{(\alpha)}_T$ and $(\tau_m, j_m)$, $m\geq 1$, become to be fixed, and the dependence on $\alpha$ will also disappear in all couple  $(\beta^{(j_m)},Y^{(j_m)})$ of the Girsanov parameters. Then we take the conditional expectation w.r.t. the processes $X^{(j_1)},\cdots, X^{(j_{k-1})}$:
$$E_{\mathbb{P}}\left [f(Z_T(k))\,|\alpha\right ]=E_{\mathbb{P}} \left [E_{\mathbb{P}}\left [f(Z_T(k))\,|\alpha,X^{(j_1)},\cdots, X^{(j_{k-1})}\right ]\right] $$
Using the scaling property of the $f$-divergence for L\'{e}vy processes together with the fact that conditionally to $\alpha , X^{(j_1)},\cdots, X^{(j_{k-1})} $
$$\frac{\xi^{(j_k)}_T}{\xi^{(j_k)}_{\tau_k}}\stackrel{d}{=}\xi^{(j_k)}_{T-\tau_k}$$
we can replace $\xi^{(j_k)}$ by  the corresponding densities for $f$-divergence minimal martingale measures $\xi^{(j_k),*}$ and this procedure will decrease $f$-divergence. Then, we take the conditional expectation w.r.t. the processes $(X^{(j_1)},\cdots , X^{(j_{m-1})},X^{(j_{m+1})},\cdots X^{(j_k)})$ and use the scaling property of $f$-divergence and the identity in law saying that conditionally to $\alpha , X^{(j_1)},\cdots , X^{(j_{m-1})},X^{(j_{m+1})},\cdots X^{(j_k)}$
$$\frac{\xi^{(j_m)}_{\tau_{m}}}{\xi^{(j_m)}_{\tau_{m-1}}}\stackrel{d}{=}\xi^{(j_m)}_{\tau_{m}-\tau_{m-1}},$$
to replace $\xi^{(j_m)}$ by $\xi^{(j_m,*)}$.
Finally, 
$$E_{\mathbb{P}}[f(\mathbb{Z}_T)]\geq \sum_{k=0}^{\infty}\,E_{\mathbb{P}}\left[I_{\{\tau_k\leq t<\tau_{k+1}\}}\,f\left(\xi^{(\alpha)}_T\,\left(\prod_{m=1}^{k-1}\frac{\xi^{(j_m),*}_{\tau_m}}{\xi^{(j_m),*}_{\tau_{m-1}}}\right)\,\frac{\xi^{(j_k),*}_T}{\xi^{(j_k),*}_{\tau_k}}\right)\right]$$
$$=E_{\mathbb{P}} \left [f(\xi^{(\alpha)}_T \,\prod_{j=1}^NZ_T^{(j),*})\right ] $$
\par 2) Now we will obtain the expression for $\xi^{(\alpha),*}$ minimizing each $f$-divergence. For logarithmic utility $f(x)= -\ln(x)-1$ we get  that
$$E_{\mathbb{P}}[f(\mathbb{Z}_T^*)]= -E_{\mathbb{P}}[\ln(\xi^{(\alpha)}_T)] -\sum_{j=1}^N\,E_{\mathbb{P}} 
[\ln ( Z^{(j),*}_T)]-1$$
By Jensen inequality we deduce that
$$-E_{\mathbb{P}}[\ln(\xi^{(\alpha)}_T)]\geq -\ln (E_{\mathbb{P}}[\xi^{(\alpha)}_T])=0,$$
so  the minimum is achieved for $\xi^{(\alpha),*}_T=1$.
\par 3) For power utility $u(x) = \frac{x^p}{p}$ we put $\gamma=\frac{p}{p-1}$. Then 
$f(x) = -\frac{1}{\gamma}x^{\gamma}$. We write
$$E_{\mathbb{P}}\left [f(\mathbb{Z}_T)\right ]\geq -\frac{1}{\gamma}E_{\mathbb{P}}\left[(\xi^{(\alpha)}_T)^{\gamma}\,\prod_{j=1}^N(Z_t^{(j),*})^{\gamma}\right]=  -\frac{1}{\gamma}E_{\mathbb{P}}\left[(\xi^{(\alpha)}_T)^{\gamma}\,\prod_{j=1}^N E_{\mathbb{P}}\left[((Z_t^{(j),*})^{\gamma}\,\bigg|\,\alpha\right]\right]$$
since conditionally to $\alpha$ the processes $Z^{(j),*}, j=1,\cdots , N,$ are independent.
Now, let us denote by $Q_T^{(j),*}$ the $f$-divergence minimal measure of the individual L\'{e}vy process $X^{(j)}$, by $P_T^{(j)}$ its law under $\mathbb{P}$ and by $H^{(j),*}_t(\gamma)$ the value of the 
corresponding Hellinger process at $t$. We recall that due to the homogeneity of the L\'evy process $X^{(j)}$ we have
$$H_t^{(j),*}{(\gamma)}= t\left[\frac{\gamma (1-\gamma)}{2}\langle c^{(j)}\beta^{(j),*},\beta^{(j),*}\rangle-\int_{\mathbb{R}^d}((Y^{(j),*}(x))^{\gamma}-\gamma Y^{(j),*}(x)-1+\gamma)\nu^{(j)}(dx)\right].$$
We denote the derivative of the Hellinger process w.r.t. $t$ by $h^{(j),*}{(\gamma)}$, i.e.
$$h^{(j),*}{(\gamma)}=\frac{\gamma (1-\gamma)}{2}\langle c^{(j)}\beta^{(j),*},\beta^{(j),*}\rangle-\int_{\mathbb{R}^d}((Y^{(j),*}(x))^{\gamma}-\gamma Y^{(j),*}(x)-1+\gamma)\nu^{(j)}(dx)$$
and we put
\begin{equation}\label{talpha}
T_j^{(\alpha)}= \int_0^T  I_{\{\alpha_{s-}=j\}}ds
\end{equation}
Using Proposition 4 we find that 
\begin{equation}\label{hellinger}
E_{\mathbb{P}}[(Z_t^{(j),*})^{\gamma}\,|\,\alpha]= \exp \left(-\int_0^T  I_{\{\alpha_{s-}=j\}}dH_s^{(j),*}(\gamma )\right)=\exp \left(-T_j^{(\alpha)}\,h^{(j),*}(\gamma )\right)
\end{equation}
and that
$$E_{\mathbb{P}}[f(\mathbb{Z}_T^*)]=-\frac{1}{\gamma}E_{\mathbb{P}}\left[(\xi^{(\alpha)}_T)^{\gamma}\,\exp \left(-  T_j^{(\alpha)}\,h^{(j),*}(\gamma )\right)\right]$$
Finally, the Lagrange method for the minimisation of the expectation under the restriction that $E_{\mathbb{P}}(\xi^{(\alpha)}_T)=1$, performed in $\omega$ by $\omega$ way, gives that
\begin{equation}\label{power}
\xi^{(\alpha),*}_T= \frac{\exp\left(\frac{1}{\gamma-1}\sum_{j=1}^N T^{(\alpha)}_j\,h^{(j),*}(\gamma )\right)}{E_{\mathbb{P}}\left[\exp\left(\frac{1}{\gamma-1}\sum_{j=1}^N T^{(\alpha)}_j\,h^{(j),*}(\gamma )\right)\right] }
\end{equation}
\par 4) For exponential utility $u(x)= 1-\exp(-x)$ with $f(x) = 1-x+x\ln(x)$ we write
$$E_{\mathbb{P}} \left [f(\xi_T^{(\alpha)}\,\prod_{j=1}^NZ_T^{(j),*})\right ]=E_{\mathbb{P}}\left [\xi_T^{(\alpha)}\ln(\xi_T^{(\alpha)})\right]+E_{\mathbb{P}}\left [\xi_T^{(\alpha)}\sum_{j=1}^N E_{\mathbb{P}}\left [Z_T^{(j),*}\ln(Z_T^{(j),*})|\alpha\right ]\right ] $$
since $E_{\mathbb{P}}\left[Z_T^{(j),*}\,|\,\alpha \right]=1$  and since the processes $Z^{(j),*}$,  conditionally to $\alpha$, are independent for $j=1,2,\cdots n$. We apply now Proposition 4 to deduce that
$$ E_{\mathbb{P}}\left[Z_T^{(j),*}\ln(Z_T^{(j),*})\,|\,\alpha\right]= T_j^{(\alpha)}\,K_T^{(j),*}$$
where $K_T^{(j),*}$ is the value of the Kulback-Leibler process at $T$ which corresponds to the L\'{e}vy process $X^{(j)}$ and defined as
$$K_T^{(j),*}= T\left[ \frac{1}{2}<c^{(j)}\beta^{(j),*},\beta^{(j),*}>+\int_{\mathbb{R}^d}[Y^{(j),*}(x)\ln(Y^{(j),*}(x) )-Y^{(j),*}(x)+1]\nu^{(j)}(dx)\right]$$
To simplify the notation we denote the derivative in $t$ of this Kulback-Leibler  process by $\kappa^{(j),*}$, i.e.
$$\kappa^{(j),*}= \frac{1}{2}<c^{(j)}\beta^{(j),*},\beta^{(j),*}>+\int_{\mathbb{R}^d}[Y^{(j),*}(x)\ln(Y^{(j),*}(x) )-Y^{(j),*}(x)+1]\nu^{(j)}(dx)$$
So, finally we have to minimize the expression
$$E_{\mathbb{P}}\left [\xi_T^{(\alpha)}\ln(\xi_T^{(\alpha)})\right ]+ E_{\mathbb{P}}\left [\xi_T^{(\alpha)}\sum_{j=1}^NT_j^{(\alpha)}\,\kappa^{(j),*}\right ]$$ under the constraint $E_{\mathbb{P}}(\xi_T^{(\alpha)})=1$. Again, using Lagrange method of the minimisation of the integrals and  performing the minimisation in $\omega$ by $\omega$ way, we get that
\begin{equation}\label{exp}
\xi^{(\alpha),*}_T= \frac{\exp\left(-\sum_{j=1}^n T^{(\alpha)}_j\,\kappa^{(j),*}\right)}{E_{\mathbb{P}}\left[\exp\left(-\sum_{j=1}^N T^{(\alpha)}_j\,\kappa^{(j),*}\right)\right] }
\end{equation}
this ends the proof of this theorem.
\end{proof}	
\section{Utility maximising strategies for L\'evy switching models}
\par In this section we will find the optimal strategies for the utility maximisation. These strategies will be automatically adapted to the progressive filtration, but for the technical reason we will search  them in the initially enlarged filtration as it was mentioned in the introduction.
\par To  find the optimal strategy we will use  Proposition  
3 and also the representation of $E_{\mathbb{Q}} \left [f'(\lambda\mathbb{Z}_T^*)\,\big|\,\hat{\mathcal{F}}_t\right]$  as a sum of the stochastic integrals using the Ito formula.
For that we  change  the measure in the conditional expectation~:
$$E_{\mathbb{Q}} \left [f'(\lambda\mathbb{Z}_T^*)\,\bigg|\,\hat{\mathcal{F}}_t\right]=E_{\mathbb{P}} \left [\frac{\mathbb{Z}_T^*}{\mathbb{Z}_t^*}f'(\lambda\mathbb{Z}_T^*)\,\bigg|\,\hat{\mathcal{F}}_t\right ]$$
Since the process $X$, conditionally to the process $\alpha$, has independent increments, 
$\frac{\mathbb{Z}_T^*}{\mathbb{Z}_t^*}$ and $\mathbb{Z}_t^*$ are also conditionally to the process $\alpha$ independent for each $t\in[0,T]$. As a conclusion,
$$E_{\mathbb{Q}} \left [f'(\lambda\mathbb{Z}_T^*)\,\bigg|\,\hat{\mathcal{F}}_t\right ]=E_{\mathbb{P}} \left [\frac{\mathbb{Z}_T^*}{\mathbb{Z}_t^*}f'(\lambda\, x\,\frac{\mathbb{Z}_T^*}{\mathbb{Z}_t^*})\,\bigg|\,\hat{\mathcal{F}}_t\right]\bigg|_{x=\mathbb{Z}_t^*}$$
\par Let us introduce the shifted process $\alpha ^{t}$ as
$$\alpha^t(s) = \alpha(s+t)$$
and let also put
\begin{equation}\label{likelihood}
\mathbb{Z}_t^*= E_{\mathbb{P}}\left [\mathbb{Z}^*_T\,\big|\,\hat{\mathcal{F}}_t\right]
\end{equation}
and 
\begin{equation}\label{za}
Z_t(\alpha)=\frac{\mathbb{Z}_t^*}{\mathbb{Z}_0^*}
\end{equation}
Then,  since the processes $X^{(1)},\cdots X^{(N)}$ and $\alpha$ are independent, and the processes $X^{(1)},\cdots X^{(N)}$ are homogeneous in time,
$$\frac{\mathbb{Z}_T^*}{\mathbb{Z}_t^*}\stackrel{d}{=} Z_{T-t}(\alpha^t)$$
Finally,
$$E_{\mathbb{Q}} \left [f'(\lambda\mathbb{Z}_T^*)\,\big|\,\hat{\mathcal{F}}_t\right]=E_{\mathbb{P}} \left[Z_{T-t}(\alpha^t)\,f'(\lambda\, x\,Z_{T-t}(\alpha^t))\,\big|\,\hat{\mathcal{F}}_t\right]\bigg|_{x=\mathbb{Z}_t^*}$$
Now, we will give the expression for the optimal strategy for each HARA utility and also the corresponding maximal expected utility.\\
To avoid the complicated notations we will omit $*$ for the Girsanov parameters $(
\beta, Y)$ corresponding to the $f$-divergence minimal martingale measure. So, for the rest of the paper we put for all $t \ge 0$ and $x \in \mathbb R^d$:
\begin{equation}
\label{form16beta}
\beta_t=\sum_{j=1}^N\beta_t^{(j),*}I_{\{\alpha_{t-}=j\}}
\end{equation} 
\begin{equation}
\label{form16y}
Y_t(x)=\sum_{j=1}^N Y_t^{(j),*}(x)I_{\{\alpha_{t-}=j\}}
\end{equation}
\subsection{When the utility is logarithmic}
Let us suppose that the utility $u(x)=\ln(x)$. Then the corresponding $f$ is also logarithm up to a constant, i.e. $f(x)=-\ln (x)-1$ with $f'(x)=-\frac{1}{x}$.
\begin{proposition}
Suppose that  for $1 \le j \le N$ the matrices $c^{(j)}$  are  invertible and $E_{\mathbb{P}}|\ln\mathbb{Z}^*_T|<\infty$. Then for the initial capital $x_0 >0$, there exists an optimal strategy and for $1 \le k \le d $ and $0\le t\le T,$ it is given by:
$$\hat \varphi_t^{(k)}=- \frac{x_0\beta^{(k)}_t}{Z_{t-}(\alpha)S_{t-}^{(k)}},$$
where $\beta$ and $Z(\alpha)$ are defined by (\ref{form16beta}) and (\ref{za}) respectively. This strategy is 
progressively adapted and the corresponding maximal expected utility is
$$U^{log}_T(x_0)=\ln(x_0)-E_{\mathbb{P}}\left [\ln\mathbb{Z}^*_T\right]=\ln(x_0)-E_{\mathbb{P}}\left [\ln(Z_T(\alpha)\right].$$
\end{proposition} 
\begin{proof}
Simple calculus shows that in considered case
$$-E_{\mathbb{P}} \left [Z_{T-t}(\alpha^t)\,f'(\lambda_0\, x\,Z_{T-t}(\alpha^t))\,\big|\,\alpha\right]=\frac{1}{\lambda_0x}$$
and from the condition (\ref{initial}) we get that $\lambda_0=\frac{1}{x_0}$ and, hence,
$$-E_{\mathbb Q}\left (f'(\lambda_0 \mathbb Z^*_T)|\hat {\mathcal F}_t\right )=\frac{x_0}{\mathbb Z^*_t}=\frac{x_0}{ Z_t(\alpha)}$$
since $\mathbb{Z}^{*}_0=\xi^{(\alpha),*}_T=1$. We have the expression
$$Z_t(\alpha)=\mathcal E(m)_t=\exp\left(m_t-\frac{1}{2}\langle m^c\rangle_t+\int_0^t\int_{\mathbb R^d} \left (\ln  (Y_s(x))-Y_s(x)+1\right )\mu_X(ds,dx)\right),$$
where $m$ is a $(\mathbb P,\hat{\mathbb{F}})$-martingale given by \eqref{mm}. From the mentioned formula for $m$ we get that
\begin{equation*}
\begin{split}
Z_t(\alpha)=\exp\bigg (&\int_0^t\,{^\top}\beta_sd X^{c,\mathbb Q}_{s} +\int_0^t\int_{\mathbb R^d}(Y_s(x)-1)(\mu_X-\nu_X^{\mathbb Q})(dt,dx)+\int_0^t\frac{1}{2}\langle c_s\beta_s,\beta_s\rangle ds\\
&\hspace{-0.5cm}+\int_0^t\int_{\mathbb R^d} (Y_s(x)-1)^2\nu_X(ds,dx)+\int_0^t\int_{\mathbb R^d} (\ln Y_s(x)-Y_s(x)+1)\mu_X(ds,dx)\bigg ),
\end{split}
\end{equation*}
where $X^{c,\mathbb Q}$ is the continuous martingale part of $X$ under $\mathbb Q$, $c$ is the density of its predictable variation and $\nu_X^{\mathbb Q}$ is the compensator of the jump measure of $X$ under~$\mathbb Q$.
\par Let us define a $\mathbb Q$-martingale $M=(M_t)_{t\ge 0}$ with 
$$M_t=\int_0^t\,{^\top}\beta_sd X^{c,\mathbb Q}_{s} +\int_0^t\int_{\mathbb R^d}\ln(Y_s(x))(\mu_X-\nu_X^{\mathbb Q})(dt,dx) $$
and also a predictable process $B=(B_t)_{t \ge 0}$ with
$$B_t=\int_0^t\frac{1}{2}\langle c_s\beta_s,\beta_s\rangle ds+\int_0^t\int_{\mathbb R^d} (Y_s(x)-1)^2\nu_X(ds,dx)$$
$$\hspace{6cm}+\int_0^t\int_{\mathbb R^d} (\ln Y_s(x)-Y_s(x)+1)\nu^{\mathbb Q}_X(ds,dx).$$
Then using the compensating formula and doing the simplifications for discontinuous martingale part we see that
\begin{equation}\label{form4z}Z_t=\exp(M_t+B_t)
\end{equation} 
Let $\tau_n=\inf\{t\geq 0 \,|\, Z_t(\alpha)\leq\frac{1}{n}\}$ with $n\geq 1$ and $\inf\{\emptyset \}=+\infty $. By It\^o formula applied to the function $g(x)= \exp(-x)$ we get
\begin{equation*}
\begin{split}
\frac{1}{Z_{t\wedge\tau_n}(\alpha)}=1&-\int_0^{t\wedge\tau_n}\frac{1}{Z_{s-}(\alpha)}(dM_s+dB_s)+\frac{1}{2}\int_0^{t\wedge\tau_n} \frac{1}{Z_{s-}(\alpha)}d\langle M^c\rangle_s\\
&\hspace{2cm}+\int_0^{t\wedge\tau_n} \int_{\mathbb R^d} \frac{1}{Z_{s-}(\alpha)}\left(Y^{-1}_s(x)-1+\ln(Y_s(x)\right)\mu_X(ds,dx).
\end{split}
\end{equation*}
Since $\tau_n\rightarrow+\infty$, we deduce after limit passage the similar expression for $\frac{1}{Z_{t}(\alpha)}$ with $t\wedge\tau_n$ replaced by $t$.
\par Using again the compensating formula and taking in account that $(Z_t(\alpha))^{-1}_{t \ge 0}$ is a $\mathbb Q$-martingale, we get that its drift part
$$
\hspace{-6cm}-\int_0^t\frac{1}{Z_{s-}(\alpha)}dB_s+\frac{1}{2}\int_0^t \frac{1}{Z_{s-}(\alpha)}d\langle M^c\rangle_s+$$
$$\hspace{5cm}\int_0^t \int_{\mathbb R^d} \frac{1}{Z_{s-}(\alpha)}\left(Y^{-1}_s(x)-1+\ln(Y_s(x)\right)\nu^{\mathbb Q}_X(ds,dx)=0.$$
According to the Proposition \ref{GR} we obtain then that for $0 \le t \le T$,
\begin{equation*}
\begin{split}
\frac{x_0}{Z_t(\alpha)}=&x_0-\int_0^t \frac{x_0}{Z_{s-}(\alpha)}{^\top}\beta_sd X^{c,\mathbb Q}_{s}+\int_0^t \int_{\mathbb R^d} \frac{x_0}{Z_{s-}(\alpha)}(Y^{-1}_s(x)-1)(\mu_X-\nu_X^{\mathbb Q})(ds,dx)\\
&\hspace{7cm}=x_0 + \sum_{k=1}^d\int_0^t \hat \varphi_s^{(k)}dS_s^{(k)}.
\end{split}
\end{equation*}
But $dS_s^{(k)}=S^{(k)}_{s-}dX^{(k)}_s$ and the last equality implies that 
$$\hspace{-5cm}\sum_{k=1}^d\int_0^t\left [ \frac{x_0}{Z_{s-}(\alpha)}\beta_s^{(k)}-\hat \varphi_s^{(k)}S_{s-}^{(k)}\right ]dX_s^{(k),c,\mathbb Q}$$
$$\hspace{3cm}=\int_0^t \int_{\mathbb R^d} \left(\frac{x_0}{Z_{s-}(\alpha)}(Y^{-1}_s(x)-1)+\sum_{k=1}^d\hat \varphi_s^{(k)}S_{s-}^{(k)}\right)(\mu_X-\nu_X^{\mathbb Q})(ds,dx).$$
In the left-hand side, we have a continuous martingale and in the right-hand side a purely discontinuous martingale, which is orthogonal to the first one. Hence, for $0 \le t \le T$,
the quadratic variation of a continuous martingale 
$$\langle \,\sum_{k=1}^d\int_0^t\left [ \frac{x_0}{Z_{s-}(\alpha)}\beta_s^{(k)}-\hat \varphi_s^{(k)}S_{s-}^{(k)}\right ]dX_s^{(k),c,\mathbb Q}\,\rangle_t=0$$
Therefore, for $0 \le t \le T$,
$$\int_0^t\sum_{k=1}^d\sum_{j=1}^d\left(\frac{x_0}{Z_{s-}(\alpha)}\beta_s^{(k)}-\hat \varphi_s^{(k)}S_{s-}^{(k)}\right) \left(\frac{x_0}{Z_{s-}(\alpha)}\beta_s^{(l)}-\hat \varphi_s^{(l)}S_{s-}^{(l)}\right)c^{(k,l)}_{s}ds=0,$$
where $c^{(k,l)}_{s}$ are the elements of the matrix $c$ defined in \eqref{form4c}.
Then, for $0\leq t\leq T$, the quadratic form corresponding to the strictly positive matrix $c$
$$\sum_{k=1}^d\sum_{j=1}^d\left (\frac{x_0}{Z_{t-}(\alpha)}\beta_t^{(k)}-\hat \varphi_t^{(k)}S_{t-}^{(k)}\right ) \left (\frac{x_0}{Z_{t-}(\alpha)}\beta_t^{(l)}-\hat \varphi_t^{(l)}S_{t-}^{(l)}\right )c^{(k,l)}_{t}=0,$$
which can  happen only at zero. Then our result for the optimal strategy follows.
\par Using again Proposition \ref{GR} and doing simple calculus, we get the formula for the optimal expected utility.
\end{proof}
\subsection{When the utility is power}
When the utility is power $u(x)=\frac{x^p}{p}$, the dual function $f$ is also power, i.e. $f(x)=-\frac{1}{\gamma}x^{\gamma}$ and $f'(x)=-x^{\gamma-1}$ with $\gamma=\frac{p}{p-1}$.
\begin{proposition}
Suppose that $E_{\mathbb P}[(Z^*_T)^{\gamma}] < \infty$ and that the matrices $c^{(j)},1\le j\le N,$ are invertible. Then for initial capital $x_0 >0$, there exists an optimal strategy and for $1 \le k \le d$ and $0 \le t \le T$, it is given by
$$\hat \varphi_t^{(k)}=\frac{x_0(\gamma-1)\,Z_{t-}^{\gamma-1}(\alpha)\,\beta_{t}^{(k)}\,\exp\left(-\sum_{j=1}^N h^{(j),*}(\gamma)\int_0^t I_{\{\alpha_{s-}=j\}}ds\right)}{S_{t-}^{(k)}}$$
with $\beta$ and $Z_t(\alpha)$  defined by (\ref{form16beta}) and (\ref{za}).
This strategy is 
progressively adapted and the corresponding maximal expected utility is
$$U^{pow}_T(x_0)= \frac{x_0^p}{p}\, \left[E_{\mathbb{P}}\exp\left(\frac{1}{\gamma -1}\sum_{j=1}^N h^{(j),*}(\gamma)T^{(\alpha)}_j\right)\right]^{\gamma}\,E_{\mathbb{P}}\exp\left(-\sum_{j=1}^N h^{(j),*}(\gamma)T^{(\alpha)}_j\right) $$
where $T^{(\alpha)}_j$ is defined by \eqref{talpha}.
\end{proposition}
\begin{proof}
In this case, 
$$-E_{\mathbb P}\left [Z_{T-t}(\alpha^t)f'(\lambda_0 x Z_{T-t}(\alpha^t)\big|\alpha   \right ]=\lambda_0^{\gamma-1}x^{\gamma-1}E_{\mathbb P}\left [Z^{\gamma}_{T-t}(\alpha^t)\big|\alpha\right].$$
The last conditional expectation, according to the discussion on the Hellinger process of the section 3, is  given by
$$E_{\mathbb P}\left [Z^{\gamma}_{T-t}(\alpha^t)\big|\alpha\right]=\exp\left(-\sum_{j=1}^N h^{(j),*}(\gamma)\int_t^TI_{\{\alpha_{s-}=j\}}ds\right).$$
Then, 
$$-E_{\mathbb Q}\left [f'(\lambda_0 \mathbb Z^*_T)\big|\hat {\mathcal F}_t\right ]=\lambda_0^{\gamma-1}(Z_t^*)^{\gamma-1}\exp\left(-\sum_{j=1}^N h^{(j),*}(\gamma)\int_t^TI_{\{\alpha_{s-}=j\}}ds\right).$$
Let us apply It\^o formula for the right hand side dropping $(\lambda_0\mathbb Z^*_0)^{\gamma-1}$. For that let $N=(N_t)_{t \ge 0}$ be a process where
$$N_t=Z_t^{\gamma-1}(\alpha)g_t$$
with $g_t=\exp\left(-\sum_{j=1}^N h^{(j),*}(\gamma)\int_t^TI_{\{\alpha_{s-}=j\}}ds\right)$.
Then, by integration by part formula
\begin{equation}\label{form4N}
N_t=g_0+\int_0^t g_s dZ_s^{\gamma-1}(\alpha)+\int_0^t Z^{\gamma-1}_s(\alpha)dg_s.
\end{equation}
since $Z_0(\alpha)=1$. Using the representation (\ref{form4z})	for $Z_s(\alpha)$ from the previous proposition, we then get that
$$Z_s^{\gamma-1}(\alpha)=\exp((\gamma-1)(M_s+B_s)).$$ 
Again  by It\^o formula via  the localisation with the stopping times
$$\tau_n=\inf\{t\geq 0 \,|\,Z_t(\alpha)\leq \frac{1}{n}\,\mbox{or}\, Z_t(\alpha)\geq n\}$$
where $n\geq 1$ and $\inf\{\emptyset \}=+\infty $, and further limit passage $n\rightarrow\infty$
we find that
\begin{equation*}
\begin{split}
Z_t^{\gamma-1}(\alpha)=&1+(\gamma-1)\int_0^tZ_{s-}^{\gamma-1}(\alpha)(dM_s+dB_s)+\frac{1}{2}(\gamma-1)^2\int_0^tZ_{s-}^{\gamma-1}(\alpha)d\langle M^c\rangle_s\\
&+\int_0^t\int_{\mathbb R^d}Z_{s-}^{\gamma-1}\left[Y_s^{\gamma-1}(x)-1-(\gamma-1)\ln Y_s(x)\right]\mu_X(ds,dx).
\end{split}
\end{equation*}
We deduce from the above formula and \eqref{form4N} that 
$$N_t=g_0+ \int_0^t Z^{\gamma-1}_s(\alpha)dg_s+$$
$$(\gamma-1)\int_0^t g_s Z_{s-}^{\gamma-1}(\alpha)dM_s+(\gamma-1)\int_0^t g_s Z_{s-}^{\gamma-1}(\alpha)dB_s+\frac{1}{2}(\gamma-1)^2\int_0^t g_sZ_{s-}^{\gamma-1}(\alpha)d\langle M^c\rangle_s$$
$$+\int_0^t\int_{\mathbb R^d}g_sZ_{s-}^{\gamma-1}\left[ Y_s^{\gamma-1}(x)-1-(\gamma-1) \ln Y_s(x)\right]\mu_X(ds,dx).$$
From the previous formula, using compensation  and taking in account that $N$ is $\mathbb Q$-martingale, we deduce  that the drift part of $N$ is equal to zero, and that 
\begin{equation*}
\begin{split}
N_t =& g_0+(\gamma-1)\int_0^t g_s Z_{s-}^{\gamma-1}(\alpha)dM_s\\
&\qquad +\int_0^t\int_{\mathbb R^d}g_sZ_{s-}^{\gamma-1}(\alpha)\left[ Y_s^{\gamma-1}(x)-1-(\gamma-1)\ln Y_s(x)\right](\mu_X-\nu_X^{\mathbb Q})(ds,dx)\\
=&g_0+(\gamma-1)\int_0^t g_s Z_{s-}^{\gamma-1}(\alpha){^{\top}}\beta_sd X^{c,\mathbb Q}_{s} + \\
&\qquad +\int_0^t\int_{\mathbb R^d}g_sZ_{s-}^{\gamma-1}\left[ Y_s^{\gamma-1}(x)-1\right](\mu_X-\nu_X^{\mathbb Q})(ds,dx)
\end{split}
\end{equation*}
According to Proposition \ref{GR}
$$(\lambda_0\mathbb Z_0^*)^{\gamma-1}N_t=x_0+\sum_{k=1}^N\int_0^t \hat \varphi_s^{(k)}S_{s-}^{(k)}dX_s^{(k)}.$$
Then $(\lambda_0\mathbb Z_0^*)^{\gamma -1}g_0=x_0$ and 
$$\lambda_0^{\gamma -1}= x_0\,\left( E_{\mathbb{P}}\left[\exp\left(\frac{1}{\gamma-1}\sum_{i=1}^N h^{(j),*}(\gamma)\int_0^TI_{\{\alpha_{s-}=j\}}ds\right)\right]\right)^{\gamma-1}$$
Moreover, for $0\leq t\leq T$
$$\sum_{k=1}^d \int_0^t\left[ \frac{x_0(\gamma-1)g_s}{g_0}Z_{s-}^{\gamma-1}(\alpha)\beta_s^{(k)}-\varphi_s^{(k)}S_{s-}^{(k)}\right]dX^{(k,c,\mathbb Q)}_s= V_t$$
where
$$V_t=\sum_{k=1}^N\int_0^t \hat \varphi_s^{(k)}S_{s-}^{(k)}dX_s^{(k),d,\mathbb Q}
- \int_0^t\int_{\mathbb R^d} \frac{x_0g_s}{g_0}Z_{s-}^{\gamma-1}(\alpha)\left[ Y_s^{\gamma-1}(x)-1\right](\mu_X-\nu_X^{\mathbb Q})(ds,dx)$$
The process $V=(V_t)_{0\leq t\leq T}$ is pure discontinuous $\mathbb Q$-martingale which is orthogonal to the continuous martingale of the left-hand side of the above equality. Then, the quadratic variation of the continuous martingale is equal to zero for $0 \le t \le T$:
$$\langle \,\sum_{k=1}^d\int_0^t\left [ \frac{x_0(\gamma-1)g_s}{g_0}Z_{s-}^{\gamma-1}(\alpha)\beta_s^{(k)}-\hat \varphi_s^{(k)}S_{s-}^{(k)}\right ]dX_s^{(k),c,\mathbb Q}\,\rangle_t=0$$
Simple calculations show that the previous equality is equivalent to
$$\int_0^t\sum_{k=1}^d\sum_{l=1}^d\left[\frac{x_0(\gamma-1)g_s}{g_0}Z_{s-}^{\gamma-1}(\alpha) \beta_s^{(k)}-\hat \varphi_s^{(k)}S_{s-}^{(k)}\right]\left[ \frac{x_0(\gamma-1)g_s}{g_0}Z_{s-}^{\gamma-1}(\alpha) \beta_s^{(l)}-\hat \varphi_s^{(l)}S_{s-}^{(l)}\right] c^{(k,l)}_{s}=0,$$
where the matrix $c_s$ is defined by (\ref{form4c}).
 Then, the  quadratic form associated with the strictly positive matrix $c_t$ 
$$\sum_{k=1}^d\sum_{l=1}^d\left[\frac{x_0(\gamma-1)g_s}{g_0}Z_{s-}^{\gamma-1}(\alpha) \beta_t^{(k)}-\hat \varphi_t^{(k)}S_{t-}^{(k)}\right]\left[\frac{x_0(\gamma-1)g_s}{g_0}Z_{s-}^{\gamma-1}(\alpha) \beta_t^{(l)}-\hat \varphi_t^{(l)}S_{t-}^{(l)}\right] c^{(k,l)}_{t}=0$$
and it  proves that 
$$\hat \varphi_t^{(k)}=\frac{x_0\,(\gamma-1)\,g_t\,Z_{t-}^{\gamma-1}(\alpha)\,\beta_{t}^{(k)}}{g_0\,S_{t-}^{(k)}}$$
and this gives the formula of the proposition.
\par Using again Proposition \ref{GR} and doing simple calculus, we get the formula for the optimal expected utility.
\end{proof}
\subsection{When the utility is exponential}
When $u$ is exponential, $f(x)=x\ln x-x+1$ and $f'(x)=\ln(x)$.
\begin{proposition}
Suppose that $E_{\mathbb P}\left[|\mathbb Z^*_T \ln \mathbb Z^*_T|\right] < \infty$ and that the matrices $c^{(j)}$ are  invertible for $1\le j\le N$. Then for the initial capital $x_0 >0$, there exists an asymptotically optimal strategy such that for $1 \le k \le d$ and $0 \le t \le T$
$$\hat \varphi^{(k)}_t=\frac{-\beta^{(k)}_t}{S_{t-}^{(k)}}$$
where $\beta$ defined by formula (\ref{form16beta}).
This strategy is progressively adapted and the corresponding maximal expected utility is  
$$U^{exp}_T(x_0)=  1-\exp(-x_0)\,\left( E_{\mathbb P}\left[\exp\left(- \sum_{j=1}^N \kappa^{(j),*}\int_0^T I_{\{\alpha_{s-}=j\}}ds)\right)\right]\right).$$
\end{proposition}
\begin{proof}
In this case,
\begin{equation*}
\begin{split}
E_{\mathbb P}\left[Z_{T-t}(\alpha^t)f'(\lambda_0 x Z_{T-t}(\alpha^t))\big|\alpha\right]=&\ln(\lambda_0x)E_{\mathbb P}\left[Z_{T-t}(\alpha^t)\big|\alpha\right]+E_{\mathbb P}\left[Z_{T-t}(\alpha^t) \ln Z_{T-t}(\alpha^t)\big|\alpha\right]\\
=& \ln (\lambda_0 x)+\sum_{j=1}^N \kappa^{(j),*}\int_t^T I_{\{\alpha_{s-}=j\}}ds
\end{split}
\end{equation*}
since $E_P(Z_{T-t}(\alpha^t)|\alpha)=1$ due to the fact that it is a martingale starting from 1. Hence,
$$-E_{\mathbb Q}(f'(\lambda_0 \mathbb Z_T^*)| \hat {\mathcal F}_t)=-\ln(\lambda_0)-\ln(\mathbb Z^*_t)-\sum_{j=1}^N \kappa^{(j),*}\int_t^T I_{\{\alpha_{s-}=j\}}ds.$$
As it was mentioned $ Z_t(\alpha)=\exp(M_t+B_t)$ w.r.t. the measure $\mathbb Q$. Hence,
\begin{equation}\label{form4lnZ}
\ln  Z_t(\alpha)=M_t+B_t.
\end{equation}
and the right-hand side of (\ref{form4lnZ}) should be a martingale. Since the processes $B=(B_t)_{t \ge 0}$ and $(\sum_{j=1}^N \int_0^t I_{\{\alpha_{s-}=j\}}dK_s^{(j)})_{t \ge 0}$ are the predictable processes, we obtain via Proposition \ref{GR} that
$$-\int_0^t \,^{T}\beta_sdX_s^{c,\mathbb Q}-\int_0^t\int_{\mathbb R^d}\ln(Y_s(x))(\mu_X-\nu_X^{\mathbb Q})(ds,dx)=\sum_{k=1}^d \int_0^t \hat \varphi_s^{(k)}S_{s-}^{(k)}dX_s^{(k)}$$
and that 
$$\ln(\lambda_0)+E_{\mathbb{P}}(\ln(\mathbb Z_0^*))+E_{\mathbb{P}}\left(\sum_{j=1}^N \kappa^{(j),*}\int_0^T I_{\{\alpha_{s-}=j\}}ds\right)=-x_0$$
 Performing the separation of continuous and discontinuous part in the right-hand side of the previous expression, and using the orthogonality of continuous and pure discontinuous martingales, we deduce  that the quadratic variation of the continuous martingale part is zero, and, hence,
$$\sum_{k=0}^d\sum_{l=0}^d\left ( \beta_t^{(k)}+\hat \varphi^{(k)}_t S_{t-}^{(k)}\right )\left ( \beta_t^{(l)}+\hat \varphi^{(l)}_t S_{t-}^{(l)}\right )c_t^{(k,l)} ds=0$$
where $c_t^{(k,l)}$ are the elements of the strictly positive matrix $c_t$ given in \eqref{form4c}.
It proves the formula for the optimal strategy.
\par Using again Proposition \ref{GR} and doing simple calculus, we get the formula for the optimal expected utility.
\end{proof}
\section{Example}
In this section, we shall consider one important example that each L\'evy process is a Brownian motion with drift. More precisely,
	$$X^{(j)}_t=b^{(j)}t+\sigma^{(j)}W^{(j)}_t,$$
	where $(W^{(j)})_{1\leq j\leq N}$ states for independent $d$-dimensional Brownian motions,  $(b^{(j)})_{1\leq j\leq N}$ are $d$-dimensional vectors and $(\sigma^{(j)})_{ 1\leq j\leq N}$ are $d\times d$ real-valued matrix. We assume that $\sigma^{(j)}$ is invertible for each $j$. In this case, the set $\mathcal M^{(j)}$ of the martingale measures consists of only one element with Girsanov parameter
	$$\beta^{(j),*}= -(^{\top}\!\sigma^{(j)}\sigma^{(j)})^{-1}b^{(j)}.$$
	Then we see that the function related to Hellinger integrals and Kulback-Leibler informations are 
	$$ h^{(j),*}(\gamma)=\frac{\gamma(1-\gamma)}{2} \langle (^{\top}\!\sigma^{(j)}\sigma^{(j)})^{-1}b^{(j)}, b^{(j)}\rangle$$
	and
	$$\kappa^{(j),*}=\frac{1}{2}\langle (^{\top}\!\sigma^{(j)}\sigma^{(j)})^{-1}b^{(j)}, b^{(j)}\rangle.$$
	As before, $\beta_s$ is defined by the formula  \eqref{form16beta}.\\
	For the power utility, from  Proposition 8, we know that the optimal investment strategy $\hat \varphi$ is given by 
	$$\hat \varphi_t^{(k)}=\frac{x_0(\gamma-1)\,Z_{t-}^{\gamma-1}(\alpha)\,\beta_{t}^{(k)}\,\exp\left(-\sum_{j=1}^N h^{(j),*}(\gamma)\int_0^t I_{\{\alpha_{s-}=j\}}ds\right)}{S_{t-}^{(k)}}$$
	where $1\leq k\leq N$ and it satisfies 
	\begin{equation}\label{optimal_investment}
(\lambda _0\mathbb{Z}^*_T)^{\gamma-1}=x_0+\sum_{k=1}^d\int_0^T  \hat{\phi}^{(k)}_s dS^{(k)}_s.
\end{equation}
	Denote by $V^{\hat \varphi}=(V_t^{\hat \varphi})_{0 \le t \le T}$ the corresponding value process of the optimal portfolio. Then
	$$V_T^{\hat \varphi}=(\lambda _0\mathbb{Z}^*_T)^{\gamma-1}.$$
	Since the stochastic integral in (\ref{optimal_investment}) is a $\mathbb Q$-martingale, we have
    $$V_t^{\hat \varphi}=E_{\mathbb Q}[(\lambda_0\mathbb Z_T^*)^{\gamma-1}|\hat {\mathcal F}_t]=\lambda_0^{\gamma-1}\mathbb Z_t^{*,\gamma-1}E_{\mathbb P}[Z^{\gamma}_{T-t}(\alpha^t)|\alpha].$$
    In the proof of Proposition 8, we have already  calculated that 
    $$V_t^{\hat \varphi}=x_0Z_{t}^{\gamma-1}(\alpha)\exp\left(-\sum_{j=1}^N h^{(j),*}(\gamma)\int_0^t I_{\{\alpha_{s-}=j\}}ds\right).$$
    Thus, the optimal strategy can also be written as 
	$$\hat \varphi^{(k)}_t =\frac{(\gamma-1)\beta_t^{(k)}V_{t-}^{\hat \varphi}}{S_{t-}^{(k)}}.$$
	With similar calculation, we can also rewrite the optimal value process for the logarithm utility in a similar form as
	$$\hat \varphi^{(k)}_t =\frac{\beta_t^{(k)}V_{t-}^{\hat \varphi}}{S_{t-}^{(k)}}.$$
	In fact, from the discussion above, we see that  the same is true for the general case where the L\'evy process admits jumps but with a non-degenerated diffusion. This suggest a very simple structure of the investment strategy. In both cases, the investor should keep a constant proportion of money in each risky assets when the market state is fixed. The proportion of money depends on the excess return and volatility. In some case with negative excess return, the investor should choose to short-sell such assets. Once the market state switched to another state, the investor should rebalance his portfolio and keep the proportion constant until next switching. In each state, we see that the total proportion that invested in the risky assets is a multiple of  $\sum_{k=1}^d \beta_t^{(k)}$. By the definition of $\beta$, we see that, comparing to bull market (high return, low volatility), the investor will put more money in the riskless account in the bear market (high volatility, low return). For the exponential utility, the strategy is similar except that the investor should keep a constant mount of money in each asset.
\section{Acknowledgements}
This research  was partially supported by Defimath project of the Research Federation of "Math\' ematiques des Pays de la Loire" and by PANORisk project "Pays de la Loire" region.

\end{document}